\documentclass{article}

\usepackage[english]{babel}

\usepackage{amsmath,amssymb,graphicx} 
\usepackage{amsthm}
\usepackage[margin=1in]{geometry} 
\usepackage{enumerate}
\usepackage{mathtools}
\usepackage{cite}
\usepackage{hyperref}
\newtheorem{theorem}{Theorem}
\newtheorem{definition}{Definition}
\newtheorem{lemma}{Lemma}
\newtheorem{observation}{Observation}
\newtheorem{remark}{Remark}

\newtheorem{example}{Example}
\newtheorem{corollary}{Corollary}
\begin{document}

\nocite{*} 

\title{Fixed points theorems for $b$-enriched multivalued nonexpansive mappings and *-$b$-enriched nonexpansive mappings}

\author{Ioan Trifoi \\
Department of Mathematics and Computer Science\\
North University Center at Baia Mare\\
Technical University of Cluj-Napoca\\
Victoriei 76, 430122 Baia Mare Romania \\
}

\date{the 27th of March 2025} 

\maketitle

\begin{abstract}

The main purpose of this paper is to extend some fixed point results for single valued $b$-enriched nonexpansive mappings to the case of multivalued mappings. 
To this end, we introduce  *-$b$-enriched nonexpansive mappings, as a generalization of *-nonexpansive mappings \cite{Abdul Rahim Khan} for which we establish an existence theorem in Hilbert space.
 We proved weak and strong convergence results of Krasnoselskii iteration process for $b$-enriched multivalued nonexpasive mappings and *-$b$-enriched nonexpansive mappings.

\end{abstract}

\section*{Introduction and preliminaries}

In the sequel, for a metric space $(X,d)$ we denote: $CB(X)$ the family of nonempty closed bounded subsets of $X$, $CC(X)$ the family of nonempty closed and convex subsets of $X$ and $K(X)$ the family of nonempty compact subsets of $X$.\\

Berinde \cite{Vasile Berinde 2010} extended the class of nonexpansive mapping into a more general class of nonexpansive mappings, named enriched nonexpansive mapping  as it follows in the below definition.\\

\begin{definition}[\cite{Vasile Berinde 2010}]
Let $(X,\|\| )$ be a linear normed space. A mapping
$T : X \longrightarrow X$ is said to be an enriched nonexpansive mapping if there
exists $b \in [0,\infty)$ such that
\begin{equation} \label{b-enriched nonexpansive mapping}
\| b(x-y)+Tx-Ty\| \leq (b+1)\|x-y\|
\end{equation}
 $\forall x,y\in X.$\par 
To indicate the constant involved in \eqref{b-enriched nonexpansive mapping} we shall also call T as a $b$-enriched nonexpansive mapping.
\end{definition}
It is easy to see that any nonexpansive mapping T is a $0$-enriched nonexpansive mapping.\\
 The specialized literature records a series of remarkable results obtained by different authors in the fixed point theory by using the enrichment  technique, to improve different classes of mappings\cite{Thabet Abdeljawad Kifayat Ullah Junaid Ahmad Muhammad Arshad Zhenhua Ma}, \cite{Sani Salisu Vasile Berinde Songpon Sriwongsa Poom Kumam}. 

Motivated by the above arguments, some remarkable results obtained for $b$-enriched nonexpansive mappings will be turned into some results for the multivalued extends of $b$-enriched nonexpansive mappings.

There are some challenges raised by turning single valued mappings into multivalued mappings: 
\begin{itemize}
\item
The metric used in multivalued mappings must cover the notion of the distance between two sets. The Pompeiu-Hausdorff metric (distance) plays a crucial role in the fixed point theory of multivalued mappings.  
\item
The multivalued mappings can be defined by transforming the single valued mapping into the multivalued mapping by replacing the single valued mapping metric with the Pompeiu-Hausdorff metric. The $b$-enriched multivalued nonexpansive mappings \cite{Mujahid Abbas Rizwan Anjum Vasile Berinde} extend the $b$-enriched nonexpansive mappings to the multivalued case by this method.
On the other hand, *-$b$-enriched nonexpansive mappings introduced in this paper, use same metric as the single valued mappings. 
\item
The multivalued mappings are more complex as the single valued mappings. The lemmas and the theorems from the single valued mappings do not hold by turning them into multivalued mappings.
\end{itemize}

We recall the point $x \in X$ is a fixed point of the multivalued mapping $T$ if $x \in Tx$. In the sequel, the set of all fixed points of $T$ will be denoted $Fix(T)$ or $F(T)$.

Abbas, Berinde and Anjum  \cite{Mujahid Abbas Rizwan Anjum Vasile Berinde}, extended the $b$-enriched nonexpansive mappings to the multivalued mappings case, by replacing the metric from Hilbert space with Pompeiu-Hausdorrf distance between two sets. On this way they defined the enriched multi-valued nonexpansive mapping as it follows:
\begin{definition} \cite{Mujahid Abbas Rizwan Anjum Vasile Berinde}
Let $(X,\|\|)$ be a liniar normed space. A multi-valued mapping $T:X \longrightarrow CB(X)$ is called $b$-enriched multivalued nonexpansive if there exists $b\in [0,\infty)$
\begin{equation}\label{benrichedmultivaluednonexpansive}
H(bx+Tx,by+Ty) \leq (b+1)\|x-y\|, \forall  x,y \in X
\end{equation}
\end{definition}

Browder's Theorem \cite{Browder FE}  plays an essential role in proving important results in the class of nonexpansive mappings.

Let $H$ be a Hibert space, $C$ a bounded, closed convex subset of $H$ and $T:C \longrightarrow C$ is a nonexpansive mapping (i.e. $\|Tx-Ty\| \leq \|x-y\|$, $\forall x, y \in C$ ).
Then for each $x_0 \in C$ and $\lambda \in [0,1)$, the mapping $T_\lambda: C \longrightarrow C$,  defined  
\begin{equation}\label{tlambdas}
T_\lambda x=(1- \lambda ) x_0+\lambda Tx , x \in C
\end{equation}
 
is a contraction on $C$. \par
According to Banach's Contraction Theorem, the mapping $T_\lambda$ has an unique fixed point denoted $x_\lambda \in C$ thus $x_\lambda = T_\lambda x_\lambda$, expressed thus:
\begin{equation}\label{xlambdas}
x_\lambda=(1- \lambda ) x_0+\lambda Tx_\lambda
\end{equation}

Browder's theorem proves that the sequence $\{x_\lambda\}$ of the fixed points of $T_\lambda$ converges to a fixed point of $T$ as $\lambda \longrightarrow 1$.\par
A problem raised by nonexpansive multivalued mappings is whether Browder's theorem can be extended from the single valued mappings to the multivalued mappings. \\
 Let consider $H$ be a Hibert space, $C$ a bounded, closed convex subset of $H$ and $T:C \longrightarrow CB(C)$ is a multivalued nonexpansive mapping (i.e. $H(Tx,Ty) \leq \|x-y\|$, $\forall x, y \in C$ ).
 
According to Nadler's Theorem \cite{Nadler 1969}, for each $x_0 \in C$, $y\in Tx$ and $\lambda \in [0,1)$, the mapping $T_\lambda: C \longrightarrow CB(C)$,  defined  
\begin{equation}\label{tlambdam}
T_\lambda x=(1- \lambda ) x_0+\lambda y , x \in C
\end{equation}
 
is a multivalued contraction mapping on $C$ and $T_\lambda$ has a (nonunique) fixed point denoted $x_\lambda \in C$ thus $x_\lambda \in T_\lambda x_\lambda$, therefore there is $y_\lambda \in T x_\lambda$ such that
\begin{equation}\label{xlambdam}
x_\lambda=(1- \lambda ) x_0+\lambda y_\lambda , x \in C, 
\end{equation}

The question is whether the sequence $\{ x_\lambda\}$ defined in \eqref{xlambdam} converges to a fixed point of $T$ as $\lambda \longrightarrow 1$.\par

Ko \cite{H M Ko} gave an example, which proves that the Browder's Theorem does not hold in general for the multivalued nonexpansive mappings.

López Acedo and Xu \cite{Genaro Lopez Hong-Kun Xu} used an additional strong condition: $Tp=\{p\}$ for each $p \in F(T)$ (singleton condition for each fixed point of multivalued nonexpansive mapping $T$), where $p$ is called endpoint or stationary point, in order to extend Browder's Theorem into multivalued nonxpansive mappings in Hilbert space. Further the other authors verified Browder's Theorem for multivalued nonexpansive type mappings using the additional singleton condition in Banach space.

The following results, due to Berinde, Abbas  and Anjum \cite{Mujahid Abbas Rizwan Anjum Vasile Berinde}, ensures the existence of fixed point  for $b$-enriched multivalued nonexpansive mappings and moreover they extended Browder's Theorem  \cite{Browder 1965} for the $b$-enriched multivalued nonexpansive mappings.

\begin{theorem}\label{Abbasexistence}\cite{Mujahid Abbas Rizwan Anjum Vasile Berinde}
Let $X$ a uniformly convex Banach space and $D$ a closed convex bounded nonempty subset of $X$. Let $T:D \rightarrow C(D)$ be a $b$-enriched multivalued nonexpansive mapping. Then $T$ has a fixed point (i.e. there exist $x\in D$ with $x \in Tx$ )
\end{theorem}

\begin{lemma}\label{fpclosedconvex}\cite{Mujahid Abbas Rizwan Anjum Vasile Berinde}
Let $D$ be a closed and convex subset of a Hilbert space $H$. Let $T:D \rightarrow C(D)$ be a $b$-enriched multivalued nonexpansive mapping with $F(T)\neq \emptyset $ and $Tp=\{p\}$ for each $p\in F(T)$. Then, $F(T)$ is closed and convex subset of $D$.
\end{lemma}

These results created the ideal premises to develop iteration schemes to approximate fixed points for the $b$-enriched multivalued nonexpansive mappings. \par

The following useful Lemma have been involved in the main section of this paper.

\begin{lemma}\label{fpTTl}(Remark1 in \cite{Mujahid Abbas Rizwan Anjum Vasile Berinde})
Let $M$ a convex subset of linear space $X$ and $T:M \rightarrow CB(M)$ given by
\[T_\lambda x=(1-\lambda)x+\lambda Tx=\{(1-\lambda)x+\lambda y:y \in Tx\}\]
For each $x \in M$, $T_\lambda x$ is a translation of the set $\lambda Tx$ by the vector $(1-\lambda )x$.Thus
\begin{equation}\label{fixedpointsTTl}
F(T_\lambda)=F(T)
\end{equation} 

\end{lemma}

 Shiro Ishikawa \cite{S Ishikawa}, established an useful property in Hilbert space $H$: 

\begin{lemma}\label{lambdapatrate}
Let $H$ be a Hilbert space and $x, y \in H$ and $\lambda \in(0,1)$ then we have:
\begin{equation}\label{lambdapatrate1}
\|\lambda x+(1-\lambda)y\|^2=\lambda\|x\|^2+(1-\lambda)\|y\|^2-\lambda(1-\lambda)\|x-y\|^2
\end{equation}
\end{lemma}

In order to prove the strong convergence of the sequences generated by Krasnoselskii iteration process, we shall use the following hemicompactness property for multivalued mappings defined in \cite{Ismat Beg AR Khan N Hussain}.
\begin{definition}
Let $C$ a subset of a Hilbert space $H$, a mapping $T:C\longrightarrow CB(C)$ is called:
hemicompact if, for each sequence $\{x_n\}$ in $C$, $\{x_n\}$ have a convergent subsequence $\{x_{n_i}\}$ of $\{x_n\}$ and $d(x_n,Tx_n)\longrightarrow 0$ as $n \longrightarrow \infty$.
\end{definition}

The followings definitions and statements are required in the proofing of the theorems from the main section of this paper.

\begin{definition}
Let $P_C(x)=\{z \in C:d(x,z)=d(x,C) \}, \forall x\in X$. Any $z \in P_C(x)$ is called a point of best approximation of x from C.
The set $C$ is called proximal (or sometime proximinal) (Chebyshev)set if $P_C(x)=\{z\}, \forall x\in X$. If $C$ is proximal, then the mapping $P_C:X \longrightarrow 2^C$ is called metric projection.
\end{definition}
\begin{definition} \cite{Abdul Rahim Khan}
Let $T$ a multivalued mapping for each $x \in C$ the multivalued mapping named the best approximation operator $P_T(x)=\{u_x \in Tx:d(x,u_x)=d(x,Tx\}$
\end{definition}
In this paper the best approximation operator $P_T(x)$ will be denoted $P_Tx$.
\begin{definition} \cite{Abdul Rahim Khan}
A single valued mapping (multivalued mapping) $f:C \longrightarrow X  (F:C \longrightarrow 2^X)$ is named selector of T if $f(x) \in Tx (Fx \subseteq Tx)$.
\end{definition}

\begin{lemma} \label{lemma closed convex set Chebyshev}
Every nonempty closed and convex subset of a strictly convex reflexive Banach
space is a Chebyshev set.
\end{lemma}
\begin{remark} \label{Banach Hilbert}
We will consider the followings useful statements:
\begin{enumerate}
\item Any weakly convergent sequence $\{x_n\}_{n=0}^\infty$ in Banach space is bounded. It is equivalent with $x_n \rightharpoonup a$ then $\|a\| \leq liminf\|x_n\|$.
\item Any uniformly convex Banach space is a strictly convex Banach space
\item Any Hilbert space is uniformly convex Banach space, also it is a most important example of uniformly convex Banach space. 

\end{enumerate}
\end{remark}

\section{Main results} \label{main}

The following theorem introduce an useful property for Krasnoselskii iteration sequence:

\begin{theorem}\label{convbenm}
Let $C$ be a closed and convex subset of a Hilbert space $H$. Let $T:C \rightarrow CB(C)$ be a $b$-enriched multivalued nonexpansive mapping with $F(T)\neq \emptyset $ and $Tp=\{p\}$ for each $p\in F(T)$. For any fixed $x_0 \in C$ and an arbitrary $\theta \in (0,1)$, define the sequence $\{x_n\}_{n=0}^ \infty$ by:
\begin{equation}\label{Krasnoselskiibenm}
x_{n+1}=(1-\theta)x_n+\theta y_n, \; n\geq 0
\end{equation}
where $y_n\in Tx_n$. Then $\lim_{n\to \infty} d(x_n, Tx_n)=0$
\end{theorem}

\begin{proof}
Let $\theta=\lambda \mu$ as $\mu=\frac{1}{b+1}$. Substituting  $b=\frac{1}{\mu}-1$ into the following inequality: 
\begin{equation}\label{benrichedmultivaluednonexpansive}
H(bx+Tx,by+Ty) \leq (b+1)\|x-y\|
\end{equation}

we get 
\[H((1-\mu)x+\mu Tx,(1-\mu)y+\mu Ty) \leq \|x-y\|\] $\forall x,y\in C$. In the following we will denote the averaged multivalued mapping $T_\mu: C\rightarrow CB(C)$ as $T_\mu x=(1-\mu)x+\mu Tx$. 
Thus the inequality \eqref{benrichedmultivaluednonexpansive} will turn into:
\[H(T_\mu x,T_\mu y)\leq\|x-y\|\] $\forall x,y \in C$\\
It is easy to observe that the averaged multivalued mapping $T_\mu$ is a multivalued  nonexpansive mapping.\\
The Krasnoselskii iteration sequence \eqref{Krasnoselskiibenm} will turn into:
\[x_{n+1}=(1-\lambda)x_n+\lambda y_{\mu n}\]
where: $\theta=\lambda \mu$, $y_{\mu n} \in T_\mu x_n$ and $\lambda \in (0,1)$ such that $y_{\mu n}=(1-\mu)x_n+\mu y_n$ .
According to Lemma \ref{fpTTl} $F(T)=F(T_\mu)$ (i.e. for each $p \in F(T)$ we have $Tp=T_\mu p=\{p\}$).
Let $x \in C$ and $p \in C$ as $Tp=T_\mu p=\{p\}$ using the Pompeiu-Hausdorff distance we get:
\begin{equation}\label{PHineg}
\begin{split}
H(T_\mu x,T_\mu p) {}&  =  max\{ \underset {y \in T_\mu x} \sup d(y,T_\mu p), \underset{p \in T_\mu p} \sup d(p,T_\mu x)  \}=max\{ \underset {y \in T_\mu x} \sup \|y-p\|, \underset{p \in T_\mu p} \sup d(p,T_\mu x)  \} \\ {}&  = \underset {y \in T_\mu x} \sup \|y-p\|  \geq \|y-p\| \; \forall y \in T_\mu x
\end{split}
\end{equation}
 In the sequel we will prove that the sequence $\{x_n\}$ is bounded.\\
From the nonexpansivness of $T_\mu$ and \eqref{PHineg} we get:

\begin{equation*}
\begin{split}
 \|x_{n+1}-p\| {}&  \leq \|(1-\lambda)(x_n-p)+\lambda (y_{\mu n}-p)\|  \leq (1-\lambda)\|x_n-p\|+\lambda \|y_{\mu n}-p\|  \\
{}& \leq (1-\lambda)\|x_n-p\|+\lambda H(T_{\mu}x_n,T_{\mu}p)\leq (1-\lambda)\|x_n-p\|+\lambda \|x_n-p\|=\|x_n-p\|
 \end{split}
\end{equation*}

That prove $\{x_n\}$ is bounded and $\lim \|x_n-p\|$ exists.\\
Applying \eqref{lambdapatrate1} from (Lemma \ref{lambdapatrate}) we obtain:

\begin{equation}\label{ineq1}
\begin{split}
\|x_{n+1}-p\|^2 {}& \leq \|(1-\lambda)(x_n-p)+\lambda (y_{\mu n}-p)\|^2  \\ {}&  \leq(1-\lambda)\|x_n-p\|^2+\lambda \|y_{\mu n}-p\|^2-\lambda(1-\lambda)\|x_n-y_{\mu n}\|^2 \\{}&  \leq (1-\lambda)\|x_n-p\|^2+\lambda H^2(T_\mu x_n,T_\mu p)-\lambda(1-\lambda)\|x_n-y_{\mu n}\|^2 \\{}&   \leq  (1-\lambda)\|x_n-p\|^2+\lambda \|x_n-p\|^2-\lambda(1-\lambda)\|x_n-y_{\mu n}\|^2  \\{}&  = \|x_n-p\|^2-\lambda(1-\lambda)\|x_n-y_{\mu n}\|^2, \; \forall n\geq 0
\end{split}
\end{equation}

The above inequality \eqref{ineq1} will be rewritten as it follows:
\begin{equation} \label{ineq2}
\lambda(1-\lambda)\|x_n-y_{\mu n}\|^2 \leq \|x_n-p\|^2-\|x_{n+1}-p\|^2 
\end{equation}
We will recall that $\lim_{n\to \infty} \|x_n-p\|=\lim_{n\to \infty} \|x_{n+1}-p\|$ for $\lambda \in (0,1)$, which together with the inequality \eqref{ineq2}  yield $\lim_{n\to \infty}\|x_n-y_{\mu n}\|=0$.
Since $y_{\mu n} \in T_\mu x_n$ we get $\lim_{n \to \infty} d(x_n,  T_{\mu}x_n)=0$.\\
On the other hand $\lim_{n\to \infty}\|x_n-y_{\mu n}\|=\lim_{n\to \infty}\|x_n-(1-\mu)x_n-\mu y_n\|=\lim_{n\to \infty}\|\mu x_n-\mu y_n\|=\lim_{n\to \infty}\|x_n- y_n\|=0$, where $y_n \in Tx_n$ we get $\lim_{n\to \infty}d(x_n, Tx_n)=0$

\end{proof}

The following theorem establish some conditions, which include the hemicompactness condition, under which Krasnoselskii iteration \eqref{Krasnoselskiibenm} converges strongly to a fixed point of $b$-enriched multivalued nonexpansive mapping:

\begin{theorem}\label{strongconvbemnm}
Let $C$ be a nonempty closed and convex subset of a Hilbert space $H$. Let $T:C \rightarrow CB(C)$ be a $b$-enriched multivalued nonexpansive mapping  such as $Tp=\{p\}$ for each $p\in F(T)$. Suppose $T$ is hemicompact and continuous with respect of Pompeiu Hausdorff metric. For any fixed $x_0 \in C$ and an arbitrary $\theta \in (0,1)$, define the sequence $\{x_n\}_{n=0}^ \infty$ by:
\begin{equation}\label{Krasnoselskiibenm02}
x_{n+1}=(1-\theta)x_n+\theta y_n, \; n\geq 0
\end{equation}
where $y_n\in Tx_n$. Then the sequence $\{x_n\}$ converges strongly to a fixed point of $T$.
\end{theorem}

\begin{proof}
In accordance with the Theorem \ref{convbenm} we have $\lim_{n\to\infty}d(x_n,Tx_n)=0$ and thus $\{x_n\}$ is bounded. Since $\{x_n\}$ is bounded and $T$ is hemicompact,  hence there exists a convergent subsequence $\{x_{n_i}\}$ of $\{x_n\}$ such as there exists some $q\in C$ as $x_{n_i}\longrightarrow q$ with $i \longrightarrow \infty$. Since $T$ is continuous with respect of Pompeiu Hausdorff metric we have $d(x_{n_i},Tx_{n_i})\longrightarrow d(q,Tq)$ with $i \longrightarrow \infty$ hence $d(q,Tq)=0$ and from closedness of $Tq$ we have $q\in F(T)$. From the proofing of the Theorem \ref{convbenm}, putting $q$ instead $p$ we get $\lim_{n\to\infty}\|x_n-q\|$ exists, following the conclusion: $\{x_n\}$ converges strongly to $q$.

\end{proof}

\begin{corollary}\label{strongconvbemnm01}
Let $K$ be a nonempty compact and convex subset of a Hilbert space $H$. Let $T:K \rightarrow CB(K)$ be a $b$-enriched multivalued nonexpansive mapping  such as $Tp=\{p\}$ for each $p\in F(T)$. Suppose $T$ is continuous with respect of Pompeiu Hausdorff metric. For any fixed $x_0 \in C$ and an arbitrary $\theta \in (0,1)$, define the sequence $\{x_n\}_{n=0}^ \infty$ by:
\begin{equation}\label{Krasnoselskiibenm021}
x_{n+1}=(1-\theta)x_n+\theta y_n, \; n\geq 0
\end{equation}
where $y_n\in Tx_n$. Then the sequence $\{x_n\}$ converges strongly to a fixed point of $T$.
\end{corollary}
\begin{proof}
The assumption $K$ is compact implies immediately the mapping $T:K \rightarrow CB(K)$ is hemicompact. The proof follows from the Theorem \ref{strongconvbemnm}.
\end{proof}

In general, without the hemicompactness condition, the strong convergence of the sequence $x_n$ generated by the Krasnoselskii iteration to a fixed point of the $b$-enriched nonexpansive mapping  cannot be ensured, but, in the next theorem there are some conditions, which ensure the weakly convergence:

\begin{theorem}\label{weaklyconvbenm}
Let $C$ be a weakly compact convex subset of a Hilbert space $H$. Let $T:C \rightarrow K(C)$ be a $b$-enriched multivalued nonexpansive mapping with $F(T)\neq \emptyset $ and $Tp=\{p\}$ for each $p\in F(T)$. For any fixed $x_0 \in C$ and an arbitrary $\theta \in (0,1)$, define the sequence $\{x_n\}_{n=0}^ \infty$ by:
\begin{equation}\label{Krasnoselskiibenm01}
x_{n+1}=(1-\theta)x_n+\theta y_n, \; n\geq 0
\end{equation}
where $y_n\in Tx_n$. Then $\{x_n\}$ converges weakly to a fixed point of $T$.
\end{theorem}

\begin{proof}
The averaged multivalued mapping $T_\mu$ defined into the proofing of Theorem \ref{convbenm} is a multivalued nonexpansive mapping. The averaged multivalued mapping's image $T_\mu x$ is a translation of the set $ \mu Tx$ by vector $(1-\mu) x$. 
From the proofing of Theorem \ref{convbenm} we have the followings: $\{x_n\}$ defined in \eqref{Krasnoselskiibenm01} is bounded, $\lim \|x_n-p\|$ exists for each $p \in F(T)$, $\lim_{n\to \infty}d(x_n, Tx_n)=0$ and $\lim_{n\to \infty}d(x_n, T_\mu x_n)=0$. According Lemma \ref{fpTTl} $F(T)=F(T_\mu)$ (i.e. $Tp=T_\mu p=\{p\}$). Since $C$ is weakly compact, there exists a subsequence $\{x_{n_i}\}$ of $x_n$ such as $x_{n_i} \rightharpoonup p_0$ for $p_0 \in C$. Suppose $p_0 \not\in Tp_0$ and $p_0 \not\in T_\mu p_0$. From the compactness of $Tp_0$ there exists $y_i\in Tp_0$ and $y_{\mu_i}\in T_\mu p_0$ such as $\|x_{n_i}-y_i\|=d(x_{n_i},Tp_0)$ as $y_i \to y \in Tp_0$,  $\|x_{n_i}-y_{\mu_i}\|=d(x_{n_i},T_\mu p_0)$ as $y_{\mu_i} \to y_\mu \in T_\mu p_0$   hence $p_0 \neq y$ and $p_0 \neq y_\mu$.  The Hilbert space $H$ have Opial's property. This property will be applied in the following inequalities:

\begin{equation*}
\begin{split}
\limsup\|x_{n_i}-y_\mu\| {}& \leq \limsup[\|x_{n_i}-y_{\mu_i}\|+\|y_{\mu_i}-y_\mu\|]=\limsup\|x_{n_i}-y_{\mu_i}\|=\limsup d(x_{n_i},T_\mu p_0)  \\{}& \leq \limsup[d(x_{n_i},T_\mu x_{n_i})+H(T_\mu x_{n_i},T_\mu p_0)],\leq \limsup \|x_{n_i}-p_0\| <\limsup\|x_{n_i}-y_\mu\|
\end{split}
\end{equation*}

This is a contradiction, hence $p_0\in T_\mu p_0$ and according to Lemma \ref{fpTTl}, in addition we obtain $p_0 \in T p_0$.\\
In the sequel we will prove $x_n \rightharpoonup p_0$. We will suppose that $\{x_n\}$ does not converge weakly to $p_0$. There exists another subsequence of $\{x_n\}$ denoted $\{x_{n_k}\}$, such as $x_{n_k}\rightharpoonup p \neq p_0$, where $p \in Tx$. Using Opial's property, which hold in Hilbert space, we have the followings:
\begin{equation*}
\begin{split}
\lim_{n\to\infty}\|x_n-p\| {}&=\limsup_{k\to\infty}\|x_{n_k}-p\|<\limsup_{k\to\infty}\|x_{n_k}-p_0\|=\limsup_{i\to\infty}\|x_{n_i}-p_0\|   \\{}&  <\limsup_{i\to\infty}\|x_{n_i}-p\|=\lim_{n\to\infty}\|x_n-p\|
\end{split}
\end{equation*}

We obtain a contradiction, hence $x_n \rightharpoonup p_0$

\end{proof}

The following example shows a class of $b$-enriched multivalued nonexpansive mapping: 

\begin{example}
Let $C=[\frac{1}{2},1]$ and a multivalued mapping $T:C\longrightarrow K(C)$ as $Tx=[\frac{1}{x},\frac{1}{x^2}]$. Let $x=\frac{1}{2}, \; y=1$. Then $T\frac{1}{2}=[2,4]$ and $T1=\{1\}$ so $1\in F(T)$. Now will prove $T$ is not a multivalued nonexpansive mapping   $H(T \frac{1}{2},T1)=max\{|2-1|,|4-1|\}=3$ result $3>|1-\frac{1}{2}|$ hence $T$ is not multivalued nonexpansive mapping and also $T$ is not multivalued quasi-nonexpansive mapping. \\
Now we will prove $T$ is a $b$-enriched multivalued nonexpansive mapping.
Let $\frac{1}{2}\leq x<y\leq 1$ thus we have $H(bx+Tx,by+ty)=H([bx+\frac{1}{x},bx+\frac{1}{x^2}],[by+\frac{1}{y},by+\frac{1}{y^2}])=max\{|bx+\frac{1}{x}-by-\frac{1}{y^2}|,|by+\frac{1}{y}-bx-\frac{1}{x^2}|\}=|bx+\frac{1}{x}-by-\frac{1}{y^2}|\}$. Using the inequality $\frac{y^2-x}{xy^2}<\frac{y^2-x^2}{x^2y^2}$, which is true for $x<1$, we obtain $|bx+\frac{1}{x}-by-\frac{1}{y^2}|\leq (b+1)|y-x|$ which is equivalent with 
\begin{equation}\label{ineqex}
|b(x-y)+\frac{y^2-x}{xy^2}|\leq |b(x-y)+\frac{y^2-x^2}{x^2y^2}|=|x-y||b-\frac{x+y}{x^2y^2}|
\end{equation}
.
 Substituing the inequality \eqref{ineqex} into inequality \eqref{benrichedmultivaluednonexpansive} (the definition of $b$-enriched multivalued nonexpansive mappings) we obtain $|x-y||b-\frac{x+y}{x^2y^2}|\leq (b+1)|y-x|$ hence $T$ is $\frac{5}{2}$-enriched multivalued nonexpansive mapping.

\end{example}

A way to avoid the Pompeiu-Hausdorff metric and the strong condition $Tp=\{p\}$ implied in $b$-enriched multivalued nonexpansive mappings, is to extend the $b$-enriched nonexpansive mapping to multivalued mapping case, by using a different definition for multivalued mapping. \par

The *-nonexpansive mapping was introduced and studied by Husain and Latif in \cite{T. Husain and A. Latif}. 

The following definition introduce the *-$b$ enriched nonexpansive mappings by extending the *-nonexpansive multivalued mappings for $b$-enriched nonexpansive mappings.

\begin{definition}
A multivalued mapping $T:C \longrightarrow 2^X, C\subset X$ is said *-$b$-enriched nonexpansive mapping if there
exists $b \in [0,\infty)$ such that for all $x,y \in C$ and $u_x \in Tx$ with $d(x,u_x)=d(x,Tx)$ there exists $u_y \in Ty$ with $d(y,u_y)=d(y,Ty)$ such as  
\begin{equation} \label{*-b-enriched nonexpansive mapping}
 \|(b(x-y)+(u_x-u_y)\|\leq (b+1)\|x-y\|
\end{equation}
 $\forall x,y\in X.$ .
\end{definition}

\begin{observation}
If $b=0$ then the multivalued mapping *-$b$-enriched nonexpansive mapping is *- nonexpansive mapping (i.e. a *- nonexpansive mapping is a *- $0$-enriched nonexpansive mapping).
\end{observation}

In the sequel, the following Lemma plays a crucial role in the results obtained  for *-$b$-enriched nonexpansive mappings:

\begin{lemma}\label{lemma projector operators}
Let $C$  a convex subset of a metric space $X$. Let $T : C \longrightarrow CC(C)$ a multivalued mapping and for any $\mu \in(0,1)$ consider the averaged multivalued mapping $T_\mu : C \longrightarrow CC(C)$ such as  $T_\mu=(1-\mu)I+\mu T$ and $Fix(T)\neq \emptyset$. Then $Fix(P_T)=Fix(P_{T_\mu})$
\end{lemma}
\begin{proof}
Let $s\in Fix(P_T)$ (i.e. $s \in P_T(s)$). According to the definition of the best approximation operator we have:  $d(s,T_\mu s)=d(s, (1-\mu)s+\mu Ts)=d(0,(1-\mu)s+\mu Ts-s)=d(0,\mu Ts-\mu s)=d(\mu s,\mu Ts)=\mu d(s,Ts)=d(s,s)=0$ thus $s\in Fix(P_{T_\mu})$ \par
Let $s\in Fix(P_{T_\mu})$. The following equalities hold:  $d(s,Ts)=\mu d(s,Ts)=d(\mu s,\mu Ts)=d(0,\mu Ts-\mu s)=d(0,(1-\mu)s+\mu Ts-s)=d(s, (1-\mu)s+\mu Ts)=d(s,T_\mu s)=d(s,s)=0$ thus $s\in Fix(P_T)$ \par
Hence we get the proving $Fix(P_T)=Fix(P_{T_\mu})$
\end{proof}

 The following Theorem extends the Theorem 3.4 provided by Abdul Rahim Khan in \cite{Abdul Rahim Khan} based on Browder's Theorem \cite{Browder FE}:

\begin{theorem}\label{closed and convex fixed point}
Let $X$ be a strictly convex Banach space and C a nonempty weakly compact  convex subset of $X$. Let $T : C \longrightarrow CC(C)$ be a *-$b$-enriched nonexpansive mapping such that $F(T)$ is nonempty. Then $F(T)$ is closed and convex.
\end{theorem}
\begin{proof}
The proof will be divided in two cases:\\
Case 1 $b \in (0,\infty)$.\\
Let $x,y \in C$ denoting $\mu=\frac{1}{b+1}$ then $b=\frac{1}{\mu}-1$ with $b \in (0,\infty)$ then $\mu \in (0,1]$.
Then by definition of *-$b$ nonexpansive mapping of mapping $T$  for each $u_x \in P_T x$ there exists $u_y \in P_T y$ such as $\|(b(x-y)+(u_x-u_y)\| \leq (b+1) \|x-y\|$. Putting $b=\frac{1}{\mu}-1$ into the inequality \eqref{*-b-enriched nonexpansive mapping}, it become: 
\begin{equation} \label{enriched1}
\|(1-\mu)(x-y)+\mu(u_x-u_y)\| \leq \|x-y\|
\end{equation} \par
We will denote $T_\mu:C \longrightarrow CC(C)$ 

\begin{equation} \label{averaged operator}
T_\mu=(1-\mu)I+\mu T
\end{equation}
as the averaged mapping.\\
  Let consider $u_{\mu x},u_{\mu y}\in C $ as it follows: $u_{\mu x}=(1-\mu)x+\mu u_x, u_{\mu y}=(1-\mu)y+\mu u_y$. Putting $u_{\mu x},u_{\mu y}\in C $ into the inequality \eqref{enriched1}, it will rewrite as it follows:
  \begin{equation}\label{nonexpansive}
   \|u_{\mu y} - u_{\mu x}\| \leq \| y-x\| 
  \end{equation}
  \par
For each $x\in C$, the set  $Tx$ is weakly compact and convex thus it is Cebysev. On the other hand since the averaged mapping's image $T_\mu x$ is a translation of the set $ \mu Tx$ by vector $(1-\mu) x$ we can easily conclude that it is weakly compact and convex and thus it is Cebysev (i.e. $P_{T_\mu} x $ have a unique element denoted $u_{\mu x}$). Thus $P_{T_\mu}$ is a single valued mapping. \par
Using definition of *- $b$ nonexpansive mappings there exists $u_{\mu_y}=P_{T_\mu}y \in T_y$ such that
\begin{equation}\label{nonexpansive projector}
\| P_{T_\mu} y-P_{T_\mu} x\|=\|u_{\mu y}-u_{\mu x}\| \leq \|y-x\|
\end{equation}
Inequality \eqref{nonexpansive projector} implies that $P_{T_\mu}$ is a nonexpansive selector of averaged mapping $T_\mu$ \eqref{averaged operator}.
By definition of $P_{T_\mu}$ for each $x \in C$ the operator $P_{T_\mu}: C \longrightarrow C$ 
\begin{equation}\label{projector distance}
d(x,P_{T_\mu} x)=d(x. u_{\mu_x})=d(x,T_\mu x)
\end{equation}
From \eqref{projector distance} results 
\begin{equation}\label{fixed points1}
F(T_\mu)=F(P_{T_\mu})
\end{equation}.\par 
From lemma\eqref{lemma projector operators} results
\begin{equation}\label{fixed points2}
F(P_{T_\mu})=F(P_T)
\end{equation} .\\
$Tx$ is a Cebysev set thus $P_T x$ is unique. The averaged mapping $P_T: C \longrightarrow C$ is a selector for $T$.
\begin{equation}\label{projector distance1}
d(x,P_T x)=d(x. u_x)=d(x,Tx)
\end{equation}
From \eqref{projector distance1} result
\begin{equation}\label{fixed points3}
F(T)=F(P_T)
\end{equation}.\par 
From \eqref{fixed points1}, \eqref{fixed points2} and \eqref{fixed points3} result 
\begin{equation}\label{fixed points4}
F(P_{T_\mu})=F(T)
\end{equation}
According Browder's Theorem \cite{Browder FE} and together with the fact that the  $P_{T_\mu}$ is a nonexpansive mapping yield $F(P_{T_\mu})$ is closed and convex set. Moreover according \eqref{fixed points4}  $F(T)$ is closed and convex, too.\\
Case 2 $b=0$ then $\mu =1$ and $T_\mu =T$. By combining \eqref{fixed points3} and the Theorem \cite{Browder FE} implies $F(T)$ is closed and convex.
\end{proof}

\begin{remark}
Using Remarks \ref{Banach Hilbert} the Theorem \ref{closed and convex fixed point} can be reformulated for Hilbert space.
\end{remark}

\begin{theorem}\label{existence theorem}
Let $H$ be a Hilbert space and C a nonempty closed convex subset of $H$. Let $T : C \longrightarrow CC(C)$ be a *-$b$-enriched nonexpansive mapping. Then $F(T)$ is nonempty. 
\end{theorem}
\begin{proof}
Let $\{ c_n\}$ a sequence of positive numbers convergent to 1 as $n \longrightarrow \infty$, $n \geq 1$ such as $0 <c_n<1$ (as example: $c_n=\frac{n-1}{n}$).\\
Let $x_0 \in C$ a given point. Let consider the averaged operator $T_\mu :C \rightarrow CC(C)$
\begin{equation} \label{averaged operator3}
T_\mu=(1-\mu)I+\mu T
\end{equation} $\mu \in(0,1)$.
Each Hilbert space is uniformly convex Banach space then for each $x\in C$, the image set $Tx$ is a weakly compact and convex thus $Tx$ and $T_\mu x$ are Cebyshev.\\
For each $x \in C$,  based on the averaged mapping $T_\mu$ and the sequence $c_n$ for each $x \in C$ we will denote the mapping $P_{T_{\mu c}}: C \longrightarrow C$ as it follows:
\begin{equation} \label{ptmuc}
P_{T_{\mu c}} x=c_n P_{T_\mu}x+(1-c_n)x_0
\end{equation}
$T_\mu$ is a *-nonexpansive mapping, then for each $x \in C$ there exists $y \in C$ such as, if we suppose $P_{T_\mu} y=u_y$ and $P_{T_\mu} y=u_y$, holds $\|u_y-u_x\| \leq \|y-x\|$. \\

Combining  definitions of mappings $P_{T_{\mu c}}$ and $T_\mu$ we obtain:\\
$\|P_{T_{\mu c}} y-P_{T_{\mu c}} x\|=\|c_n u_y+(1-c_n)x_0-c_n u_x-(1-c_n)x_0\|=c_n\|u_y-u_x\| \leq c_n\|y-x\|$ \\
This is equivalent with:
\begin{equation}\label{contractiveptmuc}
\|P_{T_{\mu c}} y-P_{T_{\mu c}} x\| \leq c_n\|y-x\|
\end{equation}
The inequality \eqref{contractiveptmuc} shows that the mapping $P_{T_{\mu c}}$ is a contractive mapping, hence $P_{T_{\mu c}}$ admit an unique fixed point denoted $u_c$.\\

The closed, convex and bounded set $C$ in Hilbert space $H$ is weakly compact. There exists a sequence $\{c_{n_j}\}$ in $(0,1)$ such that the followings statements holds: $c_{n_j} \rightarrow 1$ (as $j \rightarrow \infty$), $c_{n_j}=c_n$ and the sequence $u_{c_{n_j}}=u_{c_n}$ converges weakly to an element $p \in H$. The set $C$ is weakly compact therefore $p \in C$.\\
In the following we will proof that $p \in F(P_{T_\mu})$.

We have:\\
$\|u_{c_{n_j}}-P_{T_\mu} p\|^2=\|(u_{c_{n_j}}-p)+(p-P_{T_\mu} p)\|^2=\| u_{c_{n_j}}-p\|^2+\|p-P_{T_\mu} p\|^2+2\langle u_{c_{n_j}}-p,p-P_{T_\mu} p\rangle$\\
Because $u_{c_{n_j}}$ converges weakly to $p$ thus $u_{c_{n_j}}-p$ converges weakly to zero. This implies:\\ $\langle u_{c_{n_j}}-p,p-P_{T_\mu}\rangle \longrightarrow 0$ as $j \longrightarrow \infty$.\\
The above statements implies:
\begin{equation}\label{inequation fp}
 \lim_{j\to\infty} (\|u_{c_{n_j}}-P_{T_\mu} p\|^2-\| u_{c_{n_j}}-p\|^2)= \|p-P_{T_\mu} p\|^2
\end{equation}
On the other hand we have:\\
$\|P_{T_\mu} u_{c_{n_j}}-u_{c_{n_j}}\|=\|c_n P_{T_\mu} u_{c_{n_j}}+(1-c_n)x_0-P_{T_\mu} u_{c_{n_j}}\|=|1-c_n| \|P_{T_\mu} u_{c_{n_j}}-x_0\|$ implies\\
 $\|P_{T_\mu} u_{c_{n_j}}-u_{c_{n_j}}\| \rightarrow 0$ as $j \rightarrow \infty$
  (i.e. $\lim_{j\to\infty}(P_{T_\mu} u_{c_{n_j}}-u_{c_{n_j}})=0$)\\

$P_{T_\mu}$ is nonexpansive (i.e. $\|P_{T_\mu}  u_{c_{n_j}}-P_{T_\mu} p\| \leq\|u_{c_{n_j}}-p\|$).\\
Therefore:
\[\| u_{c_{n_j}}-P_{T_\mu}p\| \leq \|u_{c_{n_j}}+P_{T_\mu}u_{c_{n_j}}\|-\|P_{T_\mu}u_{c_{n_j}}-P_{T_\mu}p\| \leq \|u_{c_{n_j}}-P_{T_\mu} u_{c_{n_j}}\| + \| u_{c_{n_j}}-p\|\]\\
Thus
\[ \limsup (\| u_{c_{n_j}}-P_{T_\mu}p\|-\| u_{c_{n_j}}-p\|) \leq \lim_{j\to\infty}(P_{T_\mu} u_{c_{n_j}}-u_{c_{n_j}})=0\]
Since C is weakly compact we have:
 \[ \limsup (\| u_{c_{n_j}}-P_{T_\mu}p\|^2-\| u_{c_{n_j}}-p\|^2)= \limsup (\| u_{c_{n_j}}-P_{T_\mu}p\|-\| u_{c_{n_j}}-p\|)(\| u_{c_{n_j}}+P_{T_\mu}p\|-\| u_{c_{n_j}}-p\|) \leq 0 \]
hence
 \begin{equation}\label{inequation fp1}
 \lim_{j\to\infty} (\|u_{c_{n_j}}-P_{T_\mu} p\|^2-\| u_{c_{n_j}}-p\|^2)= 0
\end{equation}  
 From \eqref{inequation fp} and \eqref{inequation fp1} we obtain
 \[\|p-P_{T_\mu} p\|^2=0\].
 That implies $p$ is a fixed point of $P_{T_\mu}$
 
 From \eqref{fixed points4}  we have: 
\begin{equation}\label{fixed points5}
F(T_\mu)=F(P_{T_\mu})
\end{equation}.\par

 According \eqref{fixed points5} $p$ is a fixed point of $T_\mu$. Moreover according the Lemma \ref{lemma projector operators}, $p$ is a fixed point of $T$.

\end{proof}

\begin{remark}
The Theorem \ref{existence theorem} don't ensure the uniqueness of fixed point of multivalued *-$b$-enriched nonexpansive mapping $T$
\end{remark}

The followings theorems provide methods based on Krasnoselskii iteration, to approximate the fixed points of the *-$b$-enriched nonexpansive mappings.

\begin{theorem}\label{strong convergence}
Let $H$ be a Hilbert space and $C$ a nonempty closed convex subset of $H$. Let $T : C \longrightarrow CC(C)$ be a *-$b$-enriched nonexpansive  and hemicompact mapping. There exist $\theta \in (0,1)$ such as for any $x_0 \in C$, the Krasnoselskii iteration $\{x_n\}_{n=0}^\infty$
\begin{equation}\label{Krasnoselsii00}
x_{n+1}=(1-\theta) x_n+\theta y_n, n\geq 0, y_n \in P_T x_n
\end{equation} 
converges (strongly) to a fixed point of $T$.
\end{theorem}

\begin{proof}
Using same arguments from the proofing of the Theorem \ref{existence theorem}  yield $Tx$ is a Cebyshev set (i.e. $y_n=P_T x_n$ in \eqref{Krasnoselsii00}).
Let $\lambda \in(0,1)$ such as $\theta = \frac{\lambda}{b+1}$. The iteration process \eqref{Krasnoselsii00} is equivalent with

\begin{equation}\label{Krasnoselsii001}
x_{n+1}=(1-\frac{\lambda}{b+1}) x_n+\frac{\lambda}{b+1} P_Tx_n, n\geq0
\end{equation} 
We can rewrite as:

\begin{equation}\label{Krasnoselsii002}
x_{n+1}=(1-\lambda) x_n+\lambda(\frac{b}{b+1}I+ \frac{1}{b+1}P_T)x_n, n\geq0
\end{equation} 

Let $x,y \in C$ and $b \in (0,\infty)$. We will  denote $\mu=\frac{1}{b+1}$ yields $b=\frac{1}{\mu}-1$ as $\mu \in (0,1]$.
Hence by definition of *-$b$ nonexpansive mapping of mapping $T$,  for each $u_x \in P_T x$ there exists $u_y \in P_T y$ such as $\|(b(x-y)+(u_x-u_y)\| \leq (b+1) \|x-y\|$. Putting in \eqref{*-b-enriched nonexpansive mapping}  $b=\frac{1}{\mu}-1$ we obtain: 
\begin{equation} \label{enriched003}
\|(1-\mu)(x-y)+\mu(u_x-u_y)\| \leq \|x-y\|
\end{equation} \par
Denoting the averaged mapping:
\begin{equation} \label{averaged operator002}
P_{T_\mu}=(1-\mu)I+\mu P_T
\end{equation}
  and let $u_{\mu x}=(1-\mu)x+\mu u_x$ and $u_{\mu y}=(1-\mu)y+\mu u_y$ thus the inequality \eqref{enriched1} is equivalent with
  \begin{equation}\label{nonexpansive001}
   \|u_{\mu y} - u_{\mu x}\| \leq \| y-x\| 
  \end{equation}
  
Recall that for each $x \in C$ the images $Tx$ and $T_\mu x$ are Cebyshev sets thus the mappings $P_T$ and $P_{T_\mu}$ are single valued mappings. \\
Substituting $\mu=\frac{1}{b+1}$ and $1-\mu=\frac{b}{b+1}$ in \eqref{Krasnoselsii002} the iteration process is equivalent with:
\begin{equation}\label{Krasnoselskii003}
x_{n+1}=(1-\lambda) x_n+\lambda((1-\mu )I+ \mu P_T)x_n, n\geq0
\end{equation}
Substituting \eqref{averaged operator002} in \eqref{Krasnoselskii003} we obtain:
\begin{equation}\label{Krasnoselskii004}
x_{n+1}=(1-\lambda) x_n+\lambda P_{T_\mu}x_n, n\geq0
\end{equation}
From \eqref{averaged operator002} and \eqref{nonexpansive001} yields $P_{T_\mu}$ is nonexpansive mapping.\\ 
Now we will prove $\{x_n-P_{T_\mu}x_n\}$ converges strongly to zero.\\
Let $p\in F(P_{T_\mu})$ (i.e. $P_{T_\mu} p=p$).\\
\begin{equation}\label{ineg1}
\begin{split}
\|x_{n+1}-p\|^2 {}&=\|(1-\lambda)x_n+\lambda P_{T_\mu}x_n-p\|^2=  \|(1-\lambda)(x_n-p)+\lambda P_{T_\mu}(x_n-p)\|^2  \\{}& = (1-\lambda)^2\|x_n-p\|^2+\lambda^2\|P_{T_\mu}x_n-p\|^2+\lambda(1-\lambda)\langle P_{T_\mu}x_n-p,x_n-p \rangle \\{}& \leq \|(1-\lambda)(x_n-p)+\lambda P_{T_\mu}(x_n-p)\|^2  \\{}& =  (1-\lambda)^2\|x_n-p\|^2+\lambda^2\|x_n-p\|^2+\lambda(1-\lambda)\langle P_{T_\mu}x_n-p,x_n-p \rangle \\{}& \leq \|(1-\lambda)(x_n-p)+\lambda P_{T_\mu}(x_n-p)\|^2=  ((1-\lambda)^2+\lambda^2+\lambda(1-\lambda))\| x_n-p \|^2
\end{split} 
\end{equation}

Let $a^2 \leq \lambda(1-\lambda)$ thus

\begin{equation}\label{ineg2}
\begin{split}
a^2\|x_n-P_{T_\mu}x_n\|^2=a^2\|x_n-p\|^2+a^2\|P_{T_\mu}-p\|-2a^2\langle P_{T_\mu}x_n-p,x_n-p \rangle 
\end{split}
\end{equation}
From \eqref{ineg1} $+$ \eqref{ineg2} together with Cauchy-Schwarz inequality
\[\langle P_{T_\mu}x_n-p,x_n-p \rangle \leq |P_{T_\mu}x_n-p\|\|x_n-p\| \leq \|x_n-p\|^2\] and and $P_{T_\mu}$ nonexpansivness we obtain:
\begin{equation}\label{ineg3}
\begin{split}
\|x_{n+1}-p\|^2+a^2\|x_n-P_{T_\mu}x_n\|^2  {}& \leq (2a^2+(1-\lambda)^2+\lambda^2)\|x_n-p\|^2+2(\lambda(1-\lambda)-a^2)\langle P_{T_\mu}x_n-p,x_n-p \rangle \\{}& \leq  (2a^2+(1-\lambda)^2+\lambda^2+2\lambda(1-\lambda)-2a^2 )\|x_n-p\|^2=\|x_n-p\|^2
\end{split}
\end{equation}

Now putting in \eqref{ineg3} $a^2=\lambda(1-\lambda)$ yields:
\begin{equation}\label{ineg4}
\lambda(1-\lambda)\|x_n-P_{T_\mu}x_n\|^2 \leq \|x_n-p\|^2-\|x_{n+1}-p\|^2
\end{equation}
Summing up in \eqref{ineg4} from $n=0$ to $n=N$ we get:
\[ \lambda(1-\lambda) \sum_{n=0}^{N}\|x_n-P_{T_\mu}x_n\|^2 \leq \sum_{n=0}^{N}(\|x_n-p\|^2-\|x_{n+1}-p\|^2)=\|x_0-p\|^2-\|x_{N+1}-p\|^2 \leq \|x_0-p\|^2\]
Hence, from $\sum_{n=0}^{\infty}\|x_n-P_{T_\mu}x_n\|^2 < \infty$ follows immediately 
\begin{equation}\label{convergence}
\|x_n-P_{T_\mu}x_n\| \rightarrow 0  
\end{equation}
as n $\rightarrow \infty$ \\
If we denote $P_{T_\lambda}=(1-\lambda)I+\lambda P_{T_\mu}$ and $\theta=\lambda \mu $ we have:
\begin{equation}\label{hemicompact}
P_{T_\lambda} x-x=\lambda(P_{T_\mu} x-x)=\lambda \mu (P_Tx-x)=\theta(P_Tx-x)
\end{equation}
Since $T$ is hemicompact, from \eqref{hemicompact} follows that $P_{T}$, $P_{T_\lambda}$  and $P_{T_\mu}$  are hemicompact too.\\
$P_{T_\mu}$ is hemicompact (i.e. there exists $\{x_{ni}\}$ a strong convergent subsequence of $\{{x_n}\}$) or else (i.e. there exists $p \in C$ such that $x_{ni} \rightarrow p$)\\
Since $P_{T_\mu}$ is nonexpansive mapping hence it is continuous. From  continuity of $P_{T_\mu}$ follows that:
\begin{equation}\label{ptmucontinuity}
P_{T_\mu}x_{ni} \rightarrow P_{T_\mu}p, i \rightarrow \infty
\end{equation}
From \eqref{convergence} follows $\{x_{ni}-P_{T_\mu}x_{ni}\}$ converges strongly to $0$ and from \eqref{ptmucontinuity} follows $\{x_{ni}-P_{T_\mu}x_{ni}\}$ converges strongly to $\{p-P_{T_\mu}p\}$ proves that $p=P_{T_\mu}p$ \\
The nonexpansivness of $P_{T_\mu}$ can be expressed as follows: 
\[\|x_{n+1}-p\| \leq \|x_n-p\|, n \geq 0\]
follows the entire sequence $\{x_n\}_{n=0}^\infty$ converges strongly to $p$

According \eqref{fixed points4} from the proofing of Theorem \ref{closed and convex fixed point} $p$ is fixed point of $T$
\end{proof}

\begin{remark}\label{Picard}
The Krasnoselskii type iteration  defined in \eqref{Krasnoselsii00} for mapping $T$ is a Picard iteration for averaged mapping 
\begin{equation}\label{mapU}
U_\theta=(1-\theta)I+\theta T.
\end{equation}

\end{remark}

\begin{theorem}\label{weak convergence1}
Let $H$ be a Hilbert space and $C$ a nonempty closed convex subset of $H$. Let $T : C \longrightarrow CC(C)$ be a *-$b$-enriched nonexpansive mapping such that there exists $p \in F(T)$ as $Tp=\{p\}$. There exist $\theta \in (0,1)$ such as for any $x_0 \in C$, the Krasnoselskii iteration $\{x_n\}_{n=0}^\infty$
\begin{equation}\label{Krasnoselsii0001}
x_{n+1}=(1-\theta) x_n+\theta y_n, n\geq 0, y_n \in P_T x_n
\end{equation} 
converges (weakly) to a fixed point of $T$.

\end{theorem}

\begin{proof}
Using same arguments as in Theorem \ref{strong convergence} $T_\mu$ is *-nonexpansive mapping and $P_{T_\mu}$ is nonexpansive mapping. 
Using Lemma \ref{lemma closed convex set Chebyshev}, \eqref{fixed points3} and \eqref{fixed points4} implies $F(T)=F(P_T)=F(P_{T_\mu})=F(T_\mu)=\{p\}$.\\
We will show that $\{x_{n_j}\}_{j=0}^\infty$ given by 
\[x_{{n_j}+1}=(1-\lambda)x_{n_j}+\lambda P_{T_\mu} x_{n_j}, \: j \geq 0\]
converges weakly to a certain $p_0$, then we will prove that $p_0$ is fixed point of $T_\mu$.\\
In the following we will define $U_\lambda=(1-\lambda)I+\lambda T_\mu=(1-\theta)I+\theta T=U_\theta$, where $\theta=\lambda \mu$ and similar for projection maps $P_{U_\lambda}=(1-\lambda)I+\lambda P_{T_\mu}=(1-\theta)I+\theta P_T=P_{U_\theta}$. \\

Now we have:
\[x-P_{U_\lambda} x=x- [(1-\lambda)x+\lambda P_{T_\mu} x]=\lambda (P_{T_\mu} x - x)\].
This implies immediately that:
\begin{equation}\label{asymptotically TU}
\|x-P_{U_\lambda} x\|=\lambda \|x-P_{T_\mu} x\|
\end{equation}.

From \eqref{convergence} and \eqref{asymptotically TU} we have: 
\begin{equation}\label{asymptotically U}
\|x_{n_j}-P_{U_\lambda} x_{n_j}\| \rightarrow 0, \: as \: j\to \infty
\end{equation}

Now we will suppose $\{x_{n_j}\}_{j=0}^\infty$ does not converge weakly to $p$, which is unique according hypothesis. But we will suppose that $\{x_{n_j}\}_{j=0}^\infty$ converges weakly to a certain $p_0 \in C$.\\
From the nonexpansivness of the  mappings $P_{U_\lambda}=P_{U_\theta}$ we get: \\

\[\|x_{n_j}-P_{U_\lambda} p_0\|\leq \|P_{U_\lambda}x_{n_j}-P_{U_\lambda} p_0\|+\|x_{n_j}-P_{U_\lambda} x_{n_j}\|\leq\|x_{n_j}- p_0\|+\|x_{n_j}-P_{U_\lambda} x_{n_j}\|\] 

From this inequality together with \eqref{asymptotically U}  we obtain:
\begin{equation}\label{limsup}
\limsup (\|x_{n_j}-P_{U_\lambda} p_0\|-\|x_{n_j}- p_0\|)\leq 0
\end{equation}
Hence, since  the weak convergence of $\{x_{n_j}\}_{j=0}^\infty$ to $p_0$, which is equivalent with $\langle x_{n_j}-p_0,p_0-P_{U_\lambda} p_0 \rangle=0 $, the following equality
\[\|x_{n_j}-P_{U_\lambda} p_0\|^2=\|(x_{n_j}-p_0)+(p_0-P_{U_\lambda} p_0)\|^2=\|x_{n_j}-p_0\|^2+\|p_0-P_{U_\lambda} p_0\|^2+2\langle x_{n_j}-p_0,p_0-P_{U_\lambda} p_0 \rangle\]
implies:
\begin{equation}\label{lim}
\lim_{j\to \infty}(\|x_{n_j}-P_{U_\lambda} p_0\|^2-\|x_{n_j}-p_0\|^2)=\|p_0-P_{U_\lambda} p_0\|^2
\end{equation}
Now the left side of \eqref{lim} becomes:
\begin{equation}\label{eqdipatr}
\|x_{n_j}-P_{U_\lambda} p_0\|^2-\|x_{n_j}-p_0\|^2=(\|x_{n_j}-P_{U_\lambda} p_0\|-\|x_{n_j}-p_0\|)(\|x_{n_j}-P_{U_\lambda} p_0\|+\|x_{n_j}-p_0\|)
\end{equation}
From \eqref{limsup}, \eqref{lim} and \eqref{eqdipatr} obtain:
\[\|p_0-P_{U_\lambda} p_0\|=0\]
This is equivalent with:
\[P_{U_\lambda} p_0=p_0\]
From this equality together with \eqref{asymptotically TU} we obtain:
\[p_0\in F(P_{U_\lambda})=F(P_{T_\mu})=F(P_T)=F(T_\mu)=F(T)\]

\end{proof}

In the next theorem the condition as there exists $p \in F(T)$ as $Tp=\{p\}$ will be remove.

\begin{theorem}\label{weak convergence2}
Let $H$ be a Hilbert space and $C$ a nonempty closed convex subset of $H$. Let $T : C \longrightarrow CC(C)$ be a *-$b$-enriched nonexpansive mapping. There exists $\theta \in (0,1)$ such as for any $x_0 \in C$, the Krasnoselskii iteration $\{x_n\}_{n=0}^\infty$
\begin{equation}\label{Krasnoselskii01}
x_{n+1}=(1-\theta) x_n+\theta y_n, n\geq 0, y_n \in P_T(x_n)
\end{equation} 
converges weakly to a fixed point of $T$.
\end{theorem}

\begin{proof}

Using same arguments from theorem's \ref{existence theorem} proofing yield $Tx$ is a Cebyshev set i.e. $y_n=P_T x_n$ in \eqref{Krasnoselskii01}.
Let $\lambda \in(0,1)$ such as $\theta = \frac{\lambda}{b+1}$. The iteration process \eqref{Krasnoselskii01} is equivalent with

\begin{equation}\label{Krasnoselskii1}
x_{n+1}=(1-\frac{\lambda}{b+1}) x_n+\frac{\lambda}{b+1} P_Tx_n, n\geq0
\end{equation} 
equivalent with

\begin{equation}\label{Krasnoselskii2}
x_{n+1}=(1-\lambda) x_n+\lambda(\frac{b}{b+1}I+ \frac{1}{b+1}P_T)x_n, n\geq0
\end{equation}

  Let consider the averaged mapping $T_\mu :C \rightarrow CC(C)$
\begin{equation} \label{averaged operator3.1}
T_\mu=(1-\mu)I+\mu T, \mu \in(0,1)
\end{equation} 
As it is shown in proof of Theorem \ref{existence theorem} for each $x \in C$ the images $Tx$ and $T_\mu x$ are Cebyshev sets. Consequently the mappings $P_T$ and $P_{T_\mu}$ are single valued mappings.\\
Substituting $\mu=\frac{1}{b+1}$ and $1-\mu=\frac{b}{b+1}$ in \eqref{Krasnoselskii2} the iteration process become:
\begin{equation}\label{Krasnoselskii3}
x_{n+1}=(1-\lambda) x_n+\lambda((1-\mu )I+ \mu P_ T)x_n, n\geq0
\end{equation}
Substituting \eqref{averaged operator3.1} in \eqref{Krasnoselskii3} we obtain:

\begin{equation}\label{Krasnoselskii4}
x_{n+1}=(1-\lambda) x_n+\lambda P_{T_\mu} x_n, n\geq0
\end{equation}

From \eqref{nonexpansive001} result the mapping $T_\mu$ is *-nonexpansive mapping. Hence $P_{T_\mu}$ is nonexpansive mapping. \\
Let $F(P_{T_\mu})$ fixed points set of $P_{T_\mu}$.
The set $F(P_{T_\mu})$ is nonempty and convex. Let $p \in F(P_{T_\mu})$. From nonexpansivness of $P_{T_\mu}$ and iterative process \eqref{Krasnoselskii4} we have:

\begin{equation}\label{nonexpansive2}
\begin{split}
\|x_{n+1}-p\| {}&=\|(1-\lambda) x_n+\lambda P_{T_\mu} x_n-p\|=\|(1-\lambda) x_n-p+\lambda p+\lambda P_{T_\mu} x_n-\lambda p\|  \\{}&=\|(1-\lambda)(x_n-p)+\lambda(P_{T_\mu} x_n -p)\| =\|(1-\lambda)(x_n-p)+\lambda(P_{T_\mu} x_n -P_{T_\mu} p)\|  \\{}& \leq (1-\lambda)\|x_n-p\|+\lambda \|P_{T_\mu} x_n -P_{T_\mu} p\| \leq (1-\lambda)\|x_n-p\|+\lambda \| x_n - p\|=\|x_n-p\| 
\end{split}
\end{equation}

From \eqref{nonexpansive2} result 
\[\|x_{n+1}-p\| \leq \| x_n - p\|\]
hence the function
\[g(p)=\lim_{n \to \infty}\|x_n -p\|\]
is well defined and lower semicontinous convex function in $F(P_{T_\mu})$. Let 
\[d_0=inf\{g(p):p\in F(P_{T_\mu})\}\]
There we can be define a family of sets, thus for each $\epsilon >0$
\[F_\epsilon =\{y:g(y)\leq d_0+\epsilon \}\] closed, convex, nonempty, and bounded. In consequence each set from family is weakly compact.
Hence $\underset{\epsilon>0}  {\cap} F_\epsilon \neq \emptyset$. We will denote
\[\underset{\epsilon>0}  {\cap} F_\epsilon=\{y:g(y)=d_0\}\equiv F_0\]
Now we will proof the set $F_0$ have exactly one point.
We suppose there are two points $p_0, p_1\in F_0$. From $F_0$ convex and closed result there exists $p_\lambda=(1-\lambda)p_0+\lambda p_1, \lambda\in(0,1)$ thus as $p_\lambda \in F_0$ (i.e. $g(p_\lambda)=d_0$)
\begin{equation}\label{gpl}
\begin{split}
g^2(p_\lambda) {}& =\lim_{n \to \infty}\|p_\lambda-x_n\|^2=\lim_{n \to \infty}\|(1-\lambda)(p_0-x_n)+\lambda (p_1-x_n)\|^2  \\{}& = \lim_{n \to \infty}((1-\lambda)^2\|p_0-x_n\|^2+\lambda ^2 \|p_1-x_n)\|^2 +2\lambda(1-\lambda)\langle p_0-x_n,p_1-x_n \rangle)  \\{}& = \lim_{n \to \infty}((1-\lambda)^2\|p_0-x_n\|^2+\lambda ^2  \|p_1-x_n)\|^2+2\lambda (1-\lambda)\|p_0-x_n\| \cdot\|p_1-x_n\|)  \\{}& +2\lambda (1-\lambda)\lim_{n \to \infty}(\langle p_0-x_n,p_1-x_n \rangle-\|p_0-x_n\| \cdot \|p_1-x_n\|) \\{}& = g^2(p)+2\lambda (1-\lambda)\lim_{n \to \infty}(\langle p_0-x_n,p_1-x_n \rangle-\|p_0-x_n\| \cdot \|p_1-x_n\|)
\end{split}
\end{equation}
From \eqref{gpl} and $g(p_\lambda)=g(p)=d_0$ yields 
\begin{equation}\label{prodscal}
\lim_{n \to \infty}(\langle p_0-x_n,p_1-x_n \rangle-\|p_0-x_n\| \cdot \|p_1-x_n\|)=0
\end{equation}

From $p_0, p_1\in F_0$ hence $\|p_1-x_n\| \rightarrow d_0$, $\|p_0-x_n\| \rightarrow d_0$ and \eqref{prodscal} we get:
\[\|p_1-p_0\|^2=\|(p_1-x_n)+(x_n-p_0)\|^2=\|p_1-x_n\|^2+\|x_n-p_0\|^2-2\langle p_1-x_n, p_0-x_n\rangle \to d_0^2+d_0^2-2d_0^2=0\]\\
This proved that $F_0$ has exactly one point. The iteration sequence from \eqref{mapU} is equivalent with the following iterative process: $x_n=P_{U_\theta ^n} x_0$. Now we will show that $x_n \rightharpoonup p_0$. We will consider a subsequence $x_{n_j}$ as $j\to\infty$ of $x_n$
Will assume $x_{n_j}\rightharpoonup p$ according Theorem \ref{weak convergence1} as $p\in F(T)$. From $x_{n_j}\rightharpoonup p$ and definition of $g$ we have:\\

\[\|x_{n_j}-p_0\|^2=\|(x_{n_j}-p)-(p_0-p)\|^2=\|x_{n_j}-p\|^2+\|p-p_0\|^2-2\langle x_{n_j}-p, p_0-p\rangle\]
Using $x_{n_j}\rightharpoonup p$, which is equivalent with $\langle x_{n_j}-p, p_0-p\rangle=0$, we obtain the following:
\[g^2(p)+\|p-p_0\|^2=g^2(p_0)=d_0^2\]
This equality together with $g^2(p)\geq d_0^2$ yields:
\[\|p-p_0\|\leq 0\]
which is equivalent with conclusion $p=p_0$.

\end{proof}

\section{Conclusions}
In this paper we improve some impressed results from the class of single valued b-enriched nonexpansive mappings by turning them into the classes of multivalued b-enriched nonexpansive mappings, under the Hilbert space settings. \par
For this purpose, we introduced and studied the class of *-$b$-enriched nonexpansive mappings, by applying the enrichment method \cite{Vasile Berinde 2010} to the *-nonexpansive mappings, introduced in \cite{Abdul Rahim Khan}.\par
The challenge have been approached by turning the class of $b$-enriched nonexpansive mappings into $b$-enriched multivalued nonexpansive mappings \cite{Mujahid Abbas Rizwan Anjum Vasile Berinde} and *-$b$-enriched nonexpansive mappings.
 It is evident that multivalued nonexpansive mapping is $0$-enriched multivalued nonexpansive mapping and *-nonexpansive mapping is *-$0$-enriched nonexpansive mapping.

 The Theorem \ref{closed and convex fixed point} extends Browder's Theorem \cite{Browder FE} to the *-$b$-enriched nonexpansive mappings.
 The Theorem \ref{existence theorem} highlights the conditions to ensure the fixed point existence for *-$b$-enriched nonexpansive mappings.\par
 We used Krasanoselskii iteration method to approximate the fixed points of $b$-enriched multivalued mappings and *-$b$-enriched nonexpansive mappings. The convergence theorems establish the framework conditions, under which the Krasnoselskii sequence \eqref{Krasnoselsii00} for $b$-enriched multivalued mappings or \eqref{Krasnoselskii1} for *-$b$-enriched nonexpansive mappings converges strongly (the Theorem \ref{strongconvbemnm}, the Corrolary \ref{strongconvbemnm01} and the Theorem \ref{strong convergence})  or weakly (the Theorem \ref{weaklyconvbenm}, the Theorem \ref{weak convergence1} and the Theorem \ref{weak convergence2}).\par

 The results are obtained in the Hilbert space, opening the field to extend these results in the other spaces.

\section*{}
 
\end{document}